\documentclass{article}%
\usepackage{amsmath}
\usepackage{amsfonts}
\usepackage{amssymb}
\usepackage{graphicx}%
\providecommand{\U}[1]{\protect\rule{.1in}{.1in}}
\newtheorem{theorem}{Theorem}

\newtheorem{definition}[theorem]{Definition}

\newtheorem{proposition}[theorem]{Proposition}

\newenvironment{proof}[1][Proof]{\noindent\textbf{#1.} }{\ \rule{0.5em}{0.5em}}
\begin{document}

\title{Renormalized solutions for stochastic transport equations and the
regularization by bilinear multiplicative noise}
\author{S. Attanasio, F. Flandoli}
\maketitle

\begin{abstract}
A linear stochastic transport equation with non-regular coefficients is
considered. Under the same assumption of the deterministic theory, all weak
$L^{\infty}$-solutions are renormalized. But then, if the noise is
non-degenerate, uniqueness of weak $L^{\infty}$-solutions does not require
essential new assumptions, opposite to the deterministic case where for
instance the divergence of the drift is asked to be bounded. The proof gives a
new explanation why bilinear multiplicative noise may have a regularizing effect.

\end{abstract}

\section{Introduction}

Consider the deterministic linear transport equation in $\mathbb{R}^{d}$
\begin{equation}
\frac{\partial u}{\partial t}+\left(  b\cdot\nabla\right)  u=0,\qquad
u|_{t=0}=u_{0}\label{det transport eq}%
\end{equation}
in a non-regular framework, namely when the given vector field $b:\left[
0,T\right]  \times\mathbb{R}^{d}\rightarrow\mathbb{R}^{d}$ satisfies%
\begin{equation}
b,\operatorname{div}b\in L_{loc}^{1}\left(  \left[  0,T\right]  \times
\mathbb{R}^{d}\right)  \label{assumption 1 on b}%
\end{equation}
and the solution $u$ is of class $L^{\infty}\left(  \left[  0,T\right]
\times\mathbb{R}^{d}\right)  $, with $u_{0}\in L^{\infty}\left(
\mathbb{R}^{d}\right)  $. Di Perna and Lions \cite{DiPernaLions} have
introduced the notion of renormalized solution to this equation: it is a
solution such that
\begin{equation}
\frac{\partial\beta\left(  u\right)  }{\partial t}+\left(  b\cdot
\nabla\right)  \beta\left(  u\right)  =0\label{beta-transformed}%
\end{equation}
for all functions $\beta\in C^{1}\left(  \mathbb{R}\right)  $. When
\begin{equation}
b\in L^{1}\left(  0,T;W_{loc}^{1,1}\left(  \mathbb{R}^{d}\right)  \right)
\label{weak differentiability assumption}%
\end{equation}
a basic commutator lemma between smoothing convolution and $\left(
b\cdot\nabla\right)  $ can be proved and, as a consequence, all $L^{\infty}%
$-weak solutions are renormalized, see \cite{DiPernaLions}. This fact is
fundamental to prove uniqueness of weak solutions to equation
(\ref{det transport eq}). A main consequence is the uniqueness when the
additional conditions
\[
\frac{\left\vert b\right\vert }{1+\left\vert x\right\vert }\in L^{1}\left(
0,T;L^{\infty}\left(  \mathbb{R}^{d}\right)  \right)  ,\qquad
\operatorname{div}b\in L^{1}\left(  0,T;L^{\infty}\left(  \mathbb{R}%
^{d}\right)  \right)
\]
are fulfilled, see \cite{DiPernaLions}. These results have been generalized by
Ambrosio \cite{Ambrosio} to $BV_{loc}$-vector fields (in place of
$W_{loc}^{1,1}$). The $BV$-framework is the one adopted in the sequel, where
we make extensive use of ideas and results from \cite{Ambrosio}. The notion of
renormalized solutions has been investigated further by several authors, see
for instance \cite{AmbrosioCrippaFigalliSpinolo}, \cite{BochCrippa0},
\cite{BochudCrippa}, \cite{Figalli}, \cite{LeBrisLions0}, \cite{LeBrisLions},
\cite{Lopez et al} and many others.

Many of the previous results can be extended quite easily to a stochastic
framework of the form%
\begin{equation}
du+\left(  b\cdot\nabla\right)  udt+\sum_{k=1}^{d}\partial_{k}u\circ
dW^{k}=0,\qquad u|_{t=0}=u_{0} \label{stoch transport eq}%
\end{equation}
where $W^{k}$ are independent Brownian motions; in particular, we give below
the analogous result of renormalizability of all solutions, under the same
assumptions on $b$ as in \cite{DiPernaLions}. But the reason for developing
this extension is the fact that, after we have proved that all solutions are
renormalized, we get uniqueness in cases not covered by the classical
deterministic theory. One of our results is that, essentially, we may just get
rid of the requirement $\operatorname{div}b\left(  t,\cdot\right)  \in
L^{\infty}\left(  \mathbb{R}^{d}\right)  $ which is responsible for the
exclusion of examples like $b\left(  x\right)  =\sqrt{\left\vert x\right\vert
}$, $d=1$:

\begin{theorem}
\label{main theorem}If $b=b_{1}+b_{2}$ with

\begin{itemize}
\item $b,\operatorname{div} b\in L^{1}_{loc}([0,T]\times\mathbb{R}%
^{d}),\;b_{t}\in BV_{loc}(\mathbb{R}^{d};\mathbb{R}^{d})$ for a.e. $t\in[0,T]$

\item For some $N>d$
\[
\int^{T}_{0}\int_{\mathbb{R}^{d}} \frac{|Db_{t}|}{(1+|x|)^{N}}dxdt<\infty
\]

\item $b_{1}\in L^{2}\left(  0,T;L^{\infty}\left(  \mathbb{R}^{d}\right)
\right)  $

\item $\frac{\left\vert b_{2}\right\vert }{1+\left\vert x\right\vert }\in
L^{1}\left(  0,T;L^{\infty}\left(  \mathbb{R}^{d}\right)  \right)  ,$
$\operatorname{div}b_{2}\in L^{1}\left(  0,T;L^{\infty}\left(  \mathbb{R}%
^{d}\right)  \right)  $
\end{itemize}

then there exists a unique weak $L^{\infty}$-solution of equation
(\ref{stoch transport eq}).
\end{theorem}

One can accept a component $b_{1}$ of $b$ which has no $L^{\infty}\left(
\mathbb{R}^{d}\right)  $-control on the divergence. We included the component
$b_{2}$ in the statement to accept linear growth at infinity, but only with
$L^{\infty}$-divergence. In a sense, $b_{1}$ takes care of the irregular part
of $b$ in a bounded ball, $b_{2}$ of the more regular but possibly linear
growth part of $b$ at infinity.

That noise could improve the theory of transport equations was first
discovered by \cite{FGP}. The present work, being based on the same commutator
lemma of the deterministic case, still requires the weak differentiability
assumption (\ref{weak differentiability assumption}). On the contrary, the
approach of \cite{FGP} by stochastic characteristics allows one to get rid of
the weak differentiability of $b$. In this sense the results of \cite{FGP} are
more advanced than the present ones. However, the assumptions here and in
\cite{FGP} are not directly comparable. The main condition assumed in
\cite{FGP} is%
\[
b\in L^{\infty}\left(  0,T;C_{b}^{\alpha}\left(  \mathbb{R}^{d}\right)
\right)
\]
together with a mild integrability of $\operatorname{div}b$. Here we may
consider also discontinuous $b$ in dimension $d>1$ (in dimension 1, assumption
(\ref{weak differentiability assumption}) implies continuity). To clarify, we
give an example in section \ref{section example} which is covered here and not
by \cite{FGP}. The boundedness of $b$ was also important in \cite{FGP} to
investigate the stochastic flow, while here it is easily removed. Moreover,
the approach presented here generalizes more easily to space-dependent noise,
but we do not stress this in this paper.

A part from the technical comparison of assumptions, one of the main purposes
of this note is to describe a completely different reason, with respect to the
one given in \cite{FGP}, that explains \textit{why} noise improves the
deterministic theory. In a sense, the reason explained here is more
\textit{structural}: it may hold true for equations possibly very different
from linear transport ones, but having some common structural features. We
know at present at least another example where it works, namely the system of
infinitely many coupled equations
\begin{equation}
dX_{n}\left(  t\right)  =k_{n-1}X_{n-1}\left(  t\right)  \circ dW_{n-1}\left(
t\right)  -k_{n}X_{n+1}\left(  t\right)  \circ dW_{n}\left(  t\right)
\label{stoch dyadic eq}%
\end{equation}
with $n\geq1$, $X_{0}\left(  t\right)  =0$, $k_{0}=0$, and for instance
$k_{n}=2^{n}$. See \cite{BFM} for details. The proof in \cite{BFM} has much in
common with the one of the present paper, although at that time this
structural fact was not identified.

In a\ few sentences, the reason why Stratonovich multiplicative noise,
sometimes called Stratonovich \textit{bilinear} noise, as that of equations
(\ref{stoch transport eq}) and (\ref{stoch dyadic eq}), produces a
regularization, is the following one. When we pass from Stratonovich to
It\^{o} form, a second order differential opertator $A$ appears (see below its
form for equation (\ref{stoch transport eq}); think to a Laplacian in the
easiest case):
\[
du+\left(  b\cdot\nabla\right)  udt+\sum_{k=1}^{d}\partial_{k}u\circ
dW^{k}=\frac{1}{2}\Delta udt.
\]
This equation is equivalent to (\ref{stoch transport eq}), so there is no
regularizing effect of $\Delta$ (it is fully compensated by the It\^{o} term,
as well understood in the theory of Zakai equation of filtering). A simple way
to see that there is no regularization is to recall that the solution of
(\ref{stoch transport eq}) when $b$ is smooth (see \cite{K}) or like in
\cite{FGP} is given by%
\[
u\left(  t,x\right)  =u_{0}\left(  \varphi_{t}^{-1}\left(  x\right)  \right)
\]
for a properly defined stochastic flow $\varphi_{t}$ of diffeomrphisms, so any
irregularity of $u_{0}$ persists in time. But when we take expected value
(assume the It\^{o} term term is a martingale, thus with zero expected value)
we get the parabolic equation%
\[
\frac{dE\left[  u\right]  }{dt}+\left(  b\cdot\nabla\right)  E\left[
u\right]  =\frac{1}{2}\Delta E\left[  u\right]  .
\]
Here we have a regularizing effect. The expected value $E\left[  u\left(
t,x\right)  \right]  $ is much more regular than $u\left(  t,x\right)  $.

Unfortunately we cannot use so easily this remark to prove uniqueness: if
$u_{0}=0$, by the previous arguments we could only deduce $E\left[  u\left(
t,x\right)  \right]  =0$ (this holds under more general assumptions than those
of theorem \ref{main theorem}), which does not imply $u=0$.

But if we can prove that
\[
d\beta\left(  u\right)  +\left(  b\cdot\nabla\right)  \beta\left(  u\right)
dt+\sum_{k=1}^{d}\partial_{k}\beta\left(  u\right)  \circ dW^{k}=0
\]
for all functions $\beta\in C^{1}\left(  \mathbb{R}\right)  $, then we pass to
It\^{o} form%
\[
d\beta\left(  u\right)  +\left(  b\cdot\nabla\right)  \beta\left(  u\right)
dt+\sum_{k=1}^{d}\partial_{k}\beta\left(  u\right)  dW^{k}=\frac{1}{2}%
\Delta\beta\left(  u\right)
\]
and take expectation%
\begin{equation}
\frac{dE\left[  \beta\left(  u\right)  \right]  }{dt}+\left(  b\cdot
\nabla\right)  E\left[  \beta\left(  u\right)  \right]  =\frac{1}{2}\Delta
E\left[  \beta\left(  u\right)  \right]  .\label{parabolic equation}%
\end{equation}
Playing with positive functions $\beta$, this allows to prove $u=0$. The
advantage with respect to the deterministic case is that now we have the term
$\Delta E\left[  \beta\left(  u\right)  \right]  $, which allows us to prove
$E\left[  \beta\left(  u\right)  \right]  =0$ under more general assumptions
on $b$ than for equation (\ref{beta-transformed}). At present, the weakeness
of this method with respect to \cite{FGP} is that we need to renormalize $u$.

An idea somewhat similar to this one was told to one of the author some time
ago by B. Rozovskii, about a special variant of 3D Navier-Stokes equations.
About this, unfortunately it is clear that one limitation of this approach is
to linear equations, with deterministic coefficient $b$: the expected value
would not commute in more general cases. Indeed, for nonlinear transport-like
problems or linear with random $b$ one can give counterexamples to a claim of
regularization by noise, see \cite{FGP} and \cite{FGP-Levico}. But there are
also positive nonlinear examples, of regularization by bilinear multiplicative
noise, see \cite{BFM}, \cite{FGP-Eulero}. We are also aware of a work in
progress by A. Debussche on a stochastic version of nonlinear Schr\"{o}dinger
equations, where a special multilicative noise has a regularizing effect that
could be similar to what is described here. But each example requires special
ad hoc arguments, at present. So the structural explanation of the present
work is only a hint at the possibility that bilinear multiplicative noise
regularizes, not a general fact.

Let us finally mention that, a posteriori, we notice similarities with the
theory of stabilization by noise developed by Arnold, Crauel, Wihstutz, see
\cite{Arnold-co}, \cite{Arnold}. For a Stratonovich system written in astract
fom as%
\[
dX_{t}=BX_{t}dt+\sum_{k}C_{k}X_{t}\circ dW_{t}^{k}%
\]
the It\^{o} form is
\[
dX_{t}=\left(  B+\sum_{k}C_{k}^{2}\right)  X_{t}dt+\sum_{k}C_{k}X_{t}%
dW_{t}^{k}.
\]
There are cases when $C_{k}^{2}$ is a \textquotedblleft
negative\textquotedblright\ operator (in a sense), like when $C_{k}^{\ast
}=-C_{k}$ and $C_{k}C_{k}^{\ast}$ is positive definite. This is, in a sense,
the case of the first order differential operators $C_{k}=\partial_{k}$. When
$C_{k}^{2}$ are \textquotedblleft negative\textquotedblright, we may expect an
increase of stability, becase again%
\[
\frac{d}{dt}E\left[  X_{t}\right]  =\left(  B+\sum_{k}C_{k}^{2}\right)
E\left[  X_{t}\right]  .
\]
This is what has been proved in \cite{Arnold-co}, \cite{Arnold}, under
suitable assumptions. At the PDE level, $\left(  B+\sum_{k}C_{k}^{2}\right)  $
may be regularizing, when $B$ is not. However, going in more details, one can
prove stabilization only when the trace of $B$ is negative, see
\cite{Arnold-co}, \cite{Arnold}, not in general as the operator $\left(
B+\sum_{k}C_{k}^{2}\right)  $ would suggest. This again shows that the simple
argument about regularization of $E\left[  X_{t}\right]  $ (or $E\left[
u\right]  $ above) is only the signature of a possible but not sure
regularization of the process itself.

\section{Definitions and preliminaries\label{section stoch copy(1)}}

Consider the Stratonovich linear stochastic transport equation
(\ref{stoch transport eq}). To shorten some notation, highlight the structure
and hint at more generality (not treated here), let us define a few
differential operators. For a.e. $t\in\left[  0,T\right]  $, denote by
$A_{t},B_{t},C_{t,k}$ the linear operators from $C_{0}^{\infty}\left(
\mathbb{R}^{d}\right)  $ to $L_{loc}^{1}\left(  \mathbb{R}^{d}\right)  $
defined as
\[
\left(  B_{t}f\right)  \left(  x\right)  =\left(  b\left(  t,x\right)
\cdot\nabla\right)  f\left(  x\right)  ,\qquad\left(  C_{t,k}f\right)  \left(
x\right)  =\partial_{k}f\left(  x\right)
\]%
\[
\left(  A_{t}f\right)  \left(  x\right)  =\frac{1}{2}\sum_{k}C_{t,k}%
C_{t,k}f\left(  x\right)  ,\qquad f\in C_{0}^{\infty}\left(  \mathbb{R}%
^{d}\right)
\]
where, here, $A_{t}f=\Delta f$. Then denote by $A_{t}^{\ast},B_{t}^{\ast
},C_{t,k}^{\ast}$ their formal adjoints, again linear operators from
$C_{0}^{\infty}\left(  \mathbb{R}^{d}\right)  $ to $L_{loc}^{1}\left(
\mathbb{R}^{d}\right)  $, defined as%

\begin{align*}
\left(  B_{t}^{\ast}\varphi\right)  \left(  x\right)   &  =-\left(  b\left(
t,x\right)  \cdot\nabla\right)  \varphi\left(  x\right)  -\varphi\left(
x\right)  \operatorname{div}b\left(  t,x\right) \\
\left(  C_{t,k}^{\ast}\varphi\right)  \left(  x\right)   &  =-\partial
_{k}\varphi\left(  x\right)
\end{align*}%
\[
\left(  A_{t}^{\ast}\varphi\right)  \left(  x\right)  =\sum_{k}C_{t,k}^{\ast
}C_{t,k}^{\ast}\varphi\left(  x\right)  ,\qquad\varphi\in C_{0}^{\infty
}\left(  \mathbb{R}^{d}\right)  .
\]
If $\varphi\in C_{0}^{\infty}\left(  \mathbb{R}^{d}\right)  $ we have
$A_{\cdot}^{\ast}\varphi,B_{\cdot}^{\ast}\varphi,C_{\cdot,k}^{\ast}\varphi\in
L_{loc}^{1}\left(  \left[  0,T\right]  \times\mathbb{R}^{d}\right)  $. The
next definition requires $b,\operatorname{div}b\in L^{1}\left(  0,T;L_{loc}%
^{1}\left(  \mathbb{R}^{d}\right)  \right)  $.

\begin{definition}
\label{definition solution stoch}If $u_{0}\in L^{\infty}\left(  \mathbb{R}%
^{d}\right)  $, we say that a random field $u\left(  t,x\right)  $ is a weak
$L^{\infty}$-solution of equation (\ref{stoch transport eq}) if $u\in
L^{\infty}\left(  \Omega\times\left[  0,T\right]  \times\mathbb{R}^{d}\right)
$ and for all $\varphi\in C_{0}^{\infty}\left(  \mathbb{R}^{d}\right)  $ the
real valued process $s\mapsto\int u_{s}C_{s,k}^{\ast}\varphi dx$ has a
modification which is a continuous adapted semi-martingale and for all
$t\in\left[  0,T\right]  $ we have $P$-a.s.
\[
\int u_{t}\varphi dx+\int_{0}^{t}\left(  \int u_{s}B_{s}^{\ast}\varphi
dx\right)  ds+\sum_{k}\int_{0}^{t}\left(  \int u_{s}C_{s,k}^{\ast}\varphi
dx\right)  \circ dW_{s}^{k}=\int u_{0}\varphi dx.
\]

\end{definition}

A posteriori, form the equation itself, it follows that for all $\varphi\in
C_{0}^{\infty}\left(  \mathbb{R}^{d}\right)  $ the real valued process
$t\mapsto\int u_{t}\varphi dx$ has a continuous modification. We shall always
use it when we write $\int u_{t}\varphi dx$, $\int u_{t}B_{t}^{\ast}\varphi
dx$, $\int u_{t}C_{t,k}^{\ast}\varphi dx$.

The reason for the assumption that $\int u_{s}C_{s,k}^{\ast}\varphi dx$ is a
continuous adapted semi-martingale is that the Stratonovich integrals
\[
\int_{0}^{t}\left(  \int u_{s}C_{s,k}^{\ast}\varphi dx\right)  \circ
dW_{s}^{k}%
\]
are thus well defined and equal to the corresponding It\^{o} integrals plus
half of the joint quadratic variation:%
\[
=\int_{0}^{t}\left(  \int u_{s}C_{s,k}^{\ast}\varphi dx\right)  dW_{s}%
^{k}+\frac{1}{2}\left[  \int u_{\cdot}C_{\cdot,k}^{\ast}\varphi dx,W_{\cdot
}^{k}\right]  _{t}.
\]
Recall, to help the intuition, that (with the notation $X_{s}=\int
u_{s}C_{s,k}^{\ast}\varphi dx$)%
\[
\int_{0}^{t}X_{s}\circ dW_{s}^{k}=\lim_{n\rightarrow\infty}\sum_{t_{i}\in
\pi_{n},t_{i}\leq t}\frac{X_{t_{i+1}\wedge t}+X_{t_{i}}}{2}\left(
W_{t_{i+1}\wedge t}-W_{t_{i}}\right)
\]%
\[
\int_{0}^{t}X_{s}dW_{s}^{k}=\lim_{n\rightarrow\infty}\sum_{t_{i}\in\pi
_{n},t_{i}\leq t}X_{t_{i}}\left(  W_{t_{i+1}\wedge t}-W_{t_{i}}\right)
\]%
\[
\left[  X_{\cdot},W_{\cdot}^{k}\right]  _{t}=\lim_{n\rightarrow\infty}%
\sum_{t_{i}\in\pi_{n},t_{i}\leq t}\left(  X_{t_{i+1}\wedge t}-X_{t_{i}%
}\right)  \left(  W_{t_{i+1}\wedge t}-W_{t_{i}}\right)
\]
where $\pi_{n}$ is a sequence of finite partitions of $\left[  0,T\right]  $
with size $\left\vert \pi_{n}\right\vert \rightarrow0$ and elements
$0=t_{0}<t_{1}<...$, and the limits are in probability, uniformly in time on
compact intervals. Details about these facts can be found in Kunita \cite{K}.

\begin{proposition}
\bigskip\label{proposition Stratonovich Ito}A weak $L^{\infty}$-solution in
the previous Stratonovich sense satisfies the It\^{o} equation%
\[
\int u_{t}\varphi dx+\int_{0}^{t}\left(  \int u_{s}B_{s}^{\ast}\varphi
dx\right)  ds+\sum_{k}\int_{0}^{t}\left(  \int u_{s}C_{s,k}^{\ast}\varphi
dx\right)  dW_{s}^{k}=\int u_{0}\varphi dx+\int_{0}^{t}\left(  \int u_{s}%
A_{s}^{\ast}\varphi dx\right)  ds
\]
for all $\varphi\in C_{0}^{\infty}\left(  \mathbb{R}^{d}\right)  $.
\end{proposition}

\begin{proof}
We have only to compute $\left[  \int u_{\cdot}C_{\cdot,k}^{\ast}\varphi
dx,W_{\cdot}^{k}\right]  _{t}$. Notice that, by the equation itself,%
\[
\int u_{t}C_{t,k}^{\ast}\varphi dx+\int_{0}^{t}\left(  \int u_{s}B_{s}^{\ast
}C_{t,k}^{\ast}\varphi dx\right)  ds+\sum_{k^{\prime}}\int_{0}^{t}\left(  \int
u_{s}C_{s,k^{\prime}}^{\ast}C_{t,k}^{\ast}\varphi dx\right)  \circ
dW_{s}^{k^{\prime}}=\int u_{0}C_{t,k}^{\ast}\varphi dx.
\]

Thus, by classical rules, easily guessed by the Riemann sum approximations
recalled above, we have
\[
\left[  \int u_{\cdot}C_{\cdot,k}^{\ast}\varphi dx,W_{\cdot}^{k}\right]
_{t}=\int_{0}^{t}\left(  \int u_{s}C_{s,k}^{\ast}C_{t,k}^{\ast}\varphi
dx\right)  ds.
\]
The proof is complete, recalling the definition of $A_{t}^{\ast}$.
\end{proof}

\section{Renormalized solutions}

\begin{definition}
We say that a weak $L^{\infty}$-solution of equation (\ref{stoch transport eq}%
) is renormalized if for every $\beta\in C^{1}\left(  \mathbb{R}\right)  $ the
process $\beta\left(  u\left(  t,x\right)  \right)  $ is a weak $L^{\infty}%
$-solution of the same equation (\ref{stoch transport eq}).
\end{definition}

\begin{definition}
If $v_{0}\in L^{\infty}\left(  \mathbb{R}^{d}\right)  $, we say that $v\in
L^{\infty}\left(  \left[  0,T\right]  \times\mathbb{R}^{d}\right)  $ is a weak
$L^{\infty}$-solution of the PDE%
\[
\frac{\partial v}{\partial t}+b\cdot\nabla v=\frac{1}{2}Av,\qquad
v|_{t=0}=v_{0}%
\]
if
\[
\int v_{t}\varphi dx+\int_{0}^{t}\left(  \int v_{s}B_{s}^{\ast}\varphi
dx\right)  ds=\int v_{0}\varphi dx+\int_{0}^{t}\left(  \int v_{s}A_{s}^{\ast
}\varphi dx\right)  ds
\]
for all $\varphi\in C_{0}^{\infty}\left(  \mathbb{R}^{d}\right)  $.
\end{definition}

\begin{definition}
Let $M$ be a $n\times n$ matrix, and let $\theta\in C^{\infty}_{c}%
(\mathbb{R}^{d})$ such that $\theta\geq0$ and $\int\theta=1$. Define
\[
\Lambda(M,\theta):=\int_{\mathbb{R}^{d}}\left|  \langle Mz,\nabla\theta(z)
\rangle\right|  dz
\]
and
\[
I(\theta):=\int_{\mathbb{R}^{d}}|z||\nabla\theta(z)|dz
\]

\end{definition}

\begin{theorem}
\label{comm. ambrosio} Suppose that $b$ satisfies assumption
(\ref{assumption 1 on b}), that, for a.e. $t\in[0,T]$, $b_{t} \in
BV_{loc}(\mathbb{R}^{d})$ and that, for every compact set $Q\subset
\mathbb{R}^{d}$
\[
\int^{T}_{0} \int_{Q}|Db_{t}|dxdt<\infty
\]
Denote with $D^{s}b$ and $D^{a}b$ the singular and absolutely continuous part
of the measure $Db$ respectively, and with $M_{t}$ the rank one matrix of the
polar decomposition $D^{s}b_{t} u=M_{t}|D^{s} b_{t}|$. Let $u\in L^{\infty
}([0,T]\times\mathbb{R}^{d})$ and $\theta\in C^{\infty}_{c}(\mathbb{R}^{d})$ a
smooth even nonnegative convolution kernel, such that $\operatorname{supp}
\theta\subset B_{1}$. Define $\theta_{\varepsilon}(x)=\varepsilon^{-n}%
\theta(\frac{x}{ \varepsilon})$, $L:=\|u\|_{L^{\infty}([0,T]\times
\mathbb{R}^{d})}$ and
\[
r_{\varepsilon}:=b\cdot\nabla(u\ast\theta_{\varepsilon})-(b\cdot\nabla
u)\ast\theta_{\varepsilon}%
\]
Then, for every compact set $Q\subset\mathbb{R}^{d}$
\begin{align}
\limsup_{\varepsilon\downarrow0}\int^{T}_{0}\int_{Q}|r_{\varepsilon}|dxdt\leq
LI(\theta)|D^{s}b|([0,T]\times Q)
\end{align}
and
\begin{align}
\limsup_{\varepsilon\downarrow0}\int^{T}_{0}\int_{Q}|r_{\varepsilon}|dxdt\leq
L\int^{T}_{0}\int_{Q}\Lambda(M_{t}(x),\theta)d|Db^{s}|(t,x)+L(d+I(\theta
))|D^{a}b|([0,T]\times Q)
\end{align}
Moreover for every $\delta>0$ and vectors $\eta$ and $\zeta$, $\theta$ can be
choosen such that: $\Lambda(\eta\otimes\zeta,\theta)<\delta$.
\end{theorem}

In the sequel we will need, in addition to the estimate on $\limsup
_{\varepsilon\rightarrow0}\|r_{\varepsilon}\|_{L^{1}}(B_{R})$ given by theorem
\ref{comm. ambrosio}, an estimate on $\|r_{\varepsilon}\|_{L^{1}}$. Therefore
the following proposition will be useful.

\begin{proposition}
\label{stima epsilon} Suppose that $u\in L^{\infty}(\mathbb{R}^{d})$, $b\in
BV_{loc}(\mathbb{R}^{d};\mathbb{R}^{d})$, $\operatorname{div}b\in L^{1}%
_{loc}(\mathbb{R}^{d})$ and $\theta\in C^{\infty}_{c}(\mathbb{R}^{d})$ is a
smooth even nonnegative convolution kernel, such that $\operatorname{supp}
\theta\subset B_{1}$. Then, exists an even convolution kernel $\rho$, with
$\operatorname{supp}\rho\subset B_{1}$ such that, for every measurable
$\varphi$, it holds:
\[
\int|r_{\varepsilon}\varphi|dx\leq C_{\theta}\|u\|_{L^{\infty}(
(\operatorname{supp}\varphi)_{\varepsilon})}|Db|(|\varphi|\ast\rho
_{\varepsilon})
\]
where $(\operatorname{supp}\varphi)_{\varepsilon}=\{x\in\mathbb{R}^{d}:
\operatorname{dist}(x,\operatorname{supp}\varphi)\leq\varepsilon\}$ Therefore,
for a.e. every $x\in\mathbb{R}^{d}$ it holds:
\[
|r_{\varepsilon}|(x)\leq C_{\theta}\|u\|_{L^{\infty}(B(x,\varepsilon
))}(|Db|\ast\rho_{\varepsilon})(x)
\]

\end{proposition}

\begin{proof}
First of all note that the second inequality is an easy consequence of the
first one. From the definition of $r_{\varepsilon}$ it follows:
\[%
\begin{split}
\int| r_{\varepsilon}\varphi| dx\leq\int\int\left|  \varphi(x)u(y)\left[
\theta_{\varepsilon}(x-y)\operatorname{div}b(y) + \nabla_{y} \theta
_{\varepsilon}(x-y)\cdot\left(  b(y)-b(x)\right)  \right]  \right|  dydx\\
\leq\int\left|  u(y)\operatorname{div}b(y)\left(  \varphi\ast\theta
_{\varepsilon}\right)  (y)\right|  dy+ \int\left|  \varphi(x)\right|
\int|u(x+\varepsilon z)|\left|  \frac{b(x+\varepsilon z)-b(x)}{\varepsilon
}\cdot\nabla\theta(-z) \right|  dzdx\\
\end{split}
\]
Note that
\[
\int\left|  u(y)\operatorname{div}b(y)\left(  \varphi\ast\theta_{\varepsilon
}\right)  (y)\right|  dy\leq d \|u\|_{L^{\infty}( (\operatorname{supp}%
\varphi)_{\varepsilon})}|Db|(|\varphi\ast\theta_{\varepsilon}|)
\]
and that
\[%
\begin{split}
\int\left|  \varphi(x)\right|  \int|u(x+\varepsilon z)|\left|  \frac
{b(x+\varepsilon z)-b(x)}{\varepsilon}\cdot\nabla\theta(-z) \right|  dzdx\\
=\int_{\mathbb{R}^{d}} |\varphi(x)|\int_{\mathbb{R}^{d}} |u(x+\varepsilon
z)|\left|  \int_{\mathbb{R}}Db(x+tz)(z)\cdot\nabla\theta(-z)\left(  \frac
{1}{\varepsilon}1_{[-\varepsilon,0]}(-t)\right)  \right|  dtdzdx\\
\leq\int\int\int|\varphi(y-tz)||u(y-(\varepsilon-t)z)||Db(y)| |z||\nabla
\theta(-z)|\left(  \frac{1}{\varepsilon}1_{[-\varepsilon,0]}(-t)\right)
dtdzdy\\
\end{split}
\]
Since $\operatorname{supp}\theta\subset B_{1}$, with the change of variable
$r=zt$, we obtain:
\[%
\begin{split}
\int\int|z||\nabla\theta(-z)||\varphi(y-tz)|\left(  \frac{1}{\varepsilon
}1_{[-\varepsilon,0]}(-t)\right)  dtdz\\
\leq\|\nabla\theta\|_{\infty}\frac{1}{\varepsilon}\int^{\varepsilon}_{0}
\frac{1}{t^{d}}\int_{r\in B(0,t)}|\varphi(y-r)|drdt=C_{\theta}|\varphi
|\ast\rho^{^{\prime}}_{\varepsilon}(y)
\end{split}
\]
where $\rho^{^{\prime}}_{\varepsilon}(z)=\frac{1}{\varepsilon}\int
^{\varepsilon}_{0}\frac{1}{t^{d}}1_{|z|\leq t}dt$ is (up to a constant
independent of $\varepsilon$) an $L^{1}$ convolution kernel, with support
contained in $B_{\varepsilon}$. Therefore we have proved
\[
\int\left|  \varphi(x)\right|  \int|u(x+\varepsilon z)|\left|  \frac
{b(x+\varepsilon z)-b(x)}{\varepsilon}\cdot\nabla\theta(-z) \right|  dzdx \leq
C_{\theta}\|u\|_{L^{\infty}( (\operatorname{supp}\varphi)_{\varepsilon}%
)}|Db|(|\varphi|\ast\rho^{^{\prime}}_{\varepsilon})
\]
So, defining $\rho_{\varepsilon}=\frac{\theta_{\varepsilon}+\rho^{^{\prime}%
}_{\varepsilon}}{2}$ the proof is complete.
\end{proof}


\begin{theorem}
\label{theorem renormalized stochastic}Suppose that $b$ satisfies assumption
(\ref{assumption 1 on b}), that, for a.e. $t\in[0,T]$, $b_{t} \in
BV_{loc}(\mathbb{R}^{d})$ and that, for every compact set $Q\subset
\mathbb{R}^{d}$
\[
\int^{T}_{0} \int_{Q}|Db_{t}|dxdt<\infty
\]
Then all weak $L^{\infty}$-solution are renormalized and, for any given
$\beta\in C^{1}\left(  \mathbb{R}\right)  $, the function%
\[
v\left(  t,x\right)  =E\left[  \beta\left(  u\left(  t,x\right)  \right)
\right]
\]
is a weak $L^{\infty}$-solution of the equation%
\[
\frac{\partial v}{\partial t}+b\cdot\nabla v=\frac{1}{2}Av,\qquad
v|_{t=0}=\beta\left(  u_{0}\right)  .
\]

\end{theorem}

\begin{proof}
\textbf{Step 1} Let $u$ be a weak $L^{\infty}$ solution of equation
(\ref{stoch transport eq}). Let $\theta\in C^{\infty}_{c}(\mathbb{R}^{d})$ be
a even smooth convolution kernel, and define $\theta_{\varepsilon}%
(x):=\frac{1}{\varepsilon^{d}}\theta(\frac{x}{\varepsilon})$ and
$u_{t}^{\varepsilon}=u\ast\theta_{\varepsilon}$. Fix $y\in\mathbb{R}^{d}$, and
consider the test function $\varphi(\cdot)=\theta_{\varepsilon}(y-\cdot)$.
From the definition of week $L^{\infty}$ solution, we have:
\[
u^{\varepsilon}_{t}(y)-\int^{t}_{0} (u_{s}\operatorname{div}b_{s})\ast
\theta_{\varepsilon}(y)+(u_{t}b_{t})\ast\nabla\theta_{\varepsilon}%
(y)ds+\sum_{k=1}^{d}\int^{t}_{0}D_{k} u_{s}^{\varepsilon}(y)\circ dW^{k}%
_{s}=u^{\varepsilon}_{0}(y)
\]
Therefore, differentiating and multiplying for $\beta^{^{\prime}%
}(u^{\varepsilon}_{t})$ it holds a.s. in the sense of the distributions on
$[0,T]\times\mathbb{R}^{d}$,
\[
\frac{d}{dt}\beta(u^{\varepsilon}_{t})(y)+ b(y)\cdot\nabla\beta(u_{t}%
^{\varepsilon})(y)+\beta^{^{\prime}}(u^{\varepsilon}_{t})(y)r^{\varepsilon
}_{t}(y)+\sum_{k=1}^{d}D_{k} \beta^{^{\prime}}(u_{t}^{\varepsilon})(y)\circ
dW^{k}_{s}=0
\]
where $r^{\varepsilon}_{t}:=(b_{t}\cdot\nabla u_{t})\ast\theta_{\varepsilon
}-b_{t}\cdot\nabla( u_{t}\ast\theta_{\varepsilon})\in L^{1}_{loc}%
([0,T]\times\mathbb{R}^{d})$. So, for any $\varphi\in C^{\infty}%
_{c}(\mathbb{R}^{d})$ we have
\begin{equation}
\label{eq rinormal}%
\begin{split}
\int\beta(u^{\varepsilon}_{t})\varphi(x)dx-\int^{t}_{0}\left(  \int
\beta(u^{\varepsilon}_{s})\left[  \operatorname{div}b_{s}\varphi+b_{s}%
\cdot\nabla\varphi\right]  dx\right)  ds\\
-\sum_{k=1}^{d}\int^{t}_{0}\int\beta(u_{s}^{\varepsilon})D_{k}\varphi dx\circ
dW^{k}_{s}-\int\beta(u^{\varepsilon}_{0})\varphi dx=-\int^{t}_{0}\int
\varphi\beta^{^{\prime}}(u^{\varepsilon}_{s})r^{\varepsilon}_{s}dxds
\end{split}
\end{equation}
From the definition $u^{\varepsilon}:=u\ast\theta_{\varepsilon}$ it follows
$\beta(u^{\varepsilon}_{t})\rightarrow\beta(u_{t})$ in $L^{p}(\Omega
\times[0,T]\times B_{R})$ for every $p\in[1,\infty)$ and every $R>0$. Moreover
$\beta(u^{\varepsilon}_{t})\rightarrow\beta(u_{t})$ for a.e. $(\omega,t,x)$.
Therefore, for any sequence $\varepsilon_{n}\rightarrow0$ it is possible to
extract a subsequence still denoted by $\varepsilon_{n}$, such that for a.e.
$\omega\in\Omega$ the left hand side converge to
\[
\int\beta(u_{t})\varphi(x)dx-\int^{t}_{0}\left(  \int\beta(u_{s})\left[
\operatorname{div}b_{s}\varphi+b_{s}\cdot\nabla\varphi\right]  dx\right)  ds
-\sum_{k=1}^{d}\int^{t}_{0}\int\beta(u_{s})D_{k}\varphi dx\circ dW^{k}%
_{s}-\int\beta(u_{0})\varphi dx
\]
Therefore for a.e. $\omega\in\Omega$, $\beta^{^{\prime}}(u^{\varepsilon_{n}%
}_{t})r^{\varepsilon_{n}}_{t}$, which is uniformly bounded in $L^{1}%
([0,T]\times B_{R})$, converge to a signed measure $\sigma$ with finite total
variation on $[0,T]\times B_{R}$. So, to show that $u$ is a renormalized
solution it is sufficient to show that a.s. $\sigma=0$ on $[0,T]\times B_{R}$.
Note that the limit of the left hand side of equation (\ref{eq rinormal}) does
not depend on the choice of $\theta$ and therefore $\sigma$ does not depend on
the choice of $\theta$. Thanks to the first estimate of theorem
\ref{comm. ambrosio}, and to the boundness of $\beta^{^{\prime}}%
(u^{\varepsilon_{n}})$, $\sigma$ is a.s. singular with respect to the $d+1$
dimensional Lebesgue measure. Moreover thanks to the second estimate of
theorem \ref{comm. ambrosio}, and to the fact that $\sigma$ is singular, we
have the estimate:
\[
|\sigma|\leq\|\beta^{^{\prime}}(u)\|_{\infty}\|u\|_{\infty}\Lambda
(M_{t}(x),\theta)d|Db^{s}|(t,x)
\]
Let $g$ be the Radon-Nykodym derivative of $\sigma$ with respect to $|D^{s}%
b|$. It holds, for every smooth even nonnegative convolution kernel $\theta$,
$g\leq\|\beta^{^{\prime}}(u)\|_{\infty}\|u\|_{\infty}\Lambda(M_{t},\theta)$,
$|D^{s}b|$-a.e. Let $D\subset C^{\infty}_{c}(B_{1})$ be a countable set, dense
with respect to the norm $W^{1,1}(B_{1})$, in the set:
\[
R:=\left\{  \theta\in W^{1,1}(B_{1}):\theta\geq0,\,\theta(x)=\theta(-x)\,
\forall x\in\mathbb{R}^{d},\,\int\theta=1\right\}
\]
Being $D$ countable it holds $g(t,x)\leq\|\beta^{^{\prime}}(u)\|_{\infty
}\|u\|_{\infty}\inf_{\theta\in D}\Lambda(M_{t}(x),\theta)$ for $|D^{s}b|$-a.e.
$(t,x)$ and being $D$ dense it holds also%
\[
g(t,x)\leq\|\beta^{^{\prime}}(u)\|_{\infty}\|u\|_{\infty}\inf_{\theta\in
R}\Lambda(M_{t},\theta)
\]
for $|D^{s}b|$-a.e. $(t,x)$ Thanks to Alberti rank one theorem we know that
$M_{t}$ has rank one, and so $g=0$ and $|\sigma|=0$. \newline\textbf{Step 2}.
Thanks to the previous step it holds a.s. and for every $t\in[0,T]$ and
$\varphi\in C^{\infty}_{c}(\mathbb{R}^{d})$
\begin{equation}
\label{eq2}%
\begin{split}
\int\beta(u_{t})\varphi(x)dx-\int^{t}_{0}\left(  \int\beta(u_{s})\left[
\operatorname{div}b_{s}\varphi+b_{s}\cdot\nabla\varphi\right]  dx\right)  ds\\
-\sum_{k=1}^{d}\int^{t}_{0}\int\beta(u_{s})D_{k}\varphi dx\circ dW^{k}%
_{s}-\int\beta(u_{0})\varphi dx=0
\end{split}
\end{equation}
Applying proposition \ref{proposition Stratonovich Ito} and taking the mean value we
obtain that $v=E[\beta(u)]$ satisfies
\begin{equation}
\label{eq3}\int v_{t}\varphi(x)dx-\int^{t}_{0}\left(  \int v_{s}\left[
\operatorname{div}b_{s}\varphi+b_{s}\cdot\nabla\varphi\right]  dx\right)  ds
-\frac{1}{2}\int^{t}_{0}\int v_{s}\Delta\varphi dxds-\int v_{0}\varphi dx=0
\end{equation}
The proof is complete.
\end{proof}

\section{Proof of theorem \ref{main theorem}}

Notice that only here the strict ellipticity of the operator $\Delta$ is used
(or the non-degeneracy assumption of the coefficients $\left(  \sigma
^{k}\right)  $ in the last section), since we need parabolic regularization to
prove uniqueness without the assumption $\operatorname{div}b\left(
t,\cdot\right)  \in L^{\infty}\left(  \mathbb{R}^{d}\right)  $.

Let us make more precise a detail that was not said in the introduction. When
we say that two weak $L^{\infty}$-solutions coincide, we mean they are in the
same equivalence class of $L^{\infty}\left(  \Omega\times\left[  0,T\right]
\times\mathbb{R}^{d}\right)  $. It follows that, for every $\varphi\in
C_{0}^{\infty}\left(  \mathbb{R}^{d}\right)  $, the continuous processes $\int
u_{t}\varphi dx$ of definition \ref{definition solution stoch}\ are
indistinguishable. Let us split the proof in a few steps.

\textbf{Step 1} (from the SPDE to a parabolic PDE). Call $u$ the difference of
two solutions. It is a weak $L^{\infty}$-solution with zero initial condition.
By theorem \ref{theorem renormalized stochastic}, $u$ is renormalized and,
given $\beta_{0}\in C^{1}$, the function%
\[
v\left(  t,x\right)  =E\left[  \beta_{0}\left(  u\left(  t,x\right)  \right)
\right]
\]
is a weak $L^{\infty}$ solution of the equation%
\[
\frac{\partial v}{\partial t}+b\cdot\nabla v=\frac{1}{2}\Delta v.
\]
Choose $\beta_{0}$ such that $\beta_{0}\left(  0\right)  =0$, so $v|_{t=0}=0$,
and $\beta_{0}\left(  u\right)  >0$ for $u\neq0$. If we prove that $v_{t}=0$,
we have proved $u_{t}=0$, $P$-a.s. This easily implies that $u$ is the zero
element of $L^{\infty}\left(  \Omega\times\left[  0,T\right]  \times
\mathbb{R}^{d}\right)  $, which is our claim.

\textbf{Step 2}. (uniqueness for the parabolic equation). Define
$v_{\varepsilon}=E[\beta_{0}(u^{\varepsilon})]$ and $r^{\varepsilon}%
_{t}:=(b_{t}\cdot\nabla u_{t})\ast\theta_{\varepsilon}-b_{t}\cdot\nabla(
u_{t}\ast\theta_{\varepsilon})$. From the proof of the previous theorem we
know that $v_{\varepsilon}\rightarrow v$ a.e. and in $L^{p}_{loc}([0,T]\times
B_{R})$ for every $p\in[1,\infty)$ and every $R>0$. Therefore, to prove $v=0$,
it is sufficient to prove that, $\int\varphi v^{2}_{\varepsilon}%
dx\rightarrow0$ for a smooth and positive function $\varphi$. We will consider
the function
\[
\varphi\left(  x\right)  =\left(  1+\left\vert x\right\vert \right)  ^{-N}%
\]
where $N>d$ is the number given in the assumptions of the theorem. Note that
it holds
\[
\nabla\varphi\left(  x\right)  =-N\left(  1+\left\vert x\right\vert \right)
^{-N-1}\frac{x}{\left\vert x\right\vert }%
\]
hence%
\[
\left(  1+\left\vert x\right\vert \right)  \left\vert \nabla\varphi\left(
x\right)  \right\vert \leq N\left\vert \varphi\left(  x\right)  \right\vert
\]
From identity (\ref{eq rinormal}) of the previous theorem, using proposition
\ref{proposition Stratonovich Ito}, taking the mean value and then differentiating
and multiplying by $2v^{\varepsilon}$, we have
\begin{equation}
\label{identity time derivative}\frac{d}{dt}\int\varphi\left\vert
v_{\varepsilon}\right\vert ^{2}=-2\int\varphi v_{\varepsilon}b\cdot\nabla
v_{\varepsilon}+\int\varphi v_{\varepsilon}\Delta v_{\varepsilon}-2\int\varphi
v_{\varepsilon}E\left[  \beta_{0}^{^{\prime}}(u^{\varepsilon}_{t}%
)r^{\varepsilon}_{t}\right]
\end{equation}
for every $\varphi\in C^{\infty}_{c}(\mathbb{R}^{d})$. Using the boundedness
of $v_{\varepsilon},\nabla v_{\varepsilon},\Delta v_{\varepsilon}$ and the
integrability over $\mathbb{R}^{d}$ of $\varphi$ (see also the next step for
the finiteness of the term $\int\varphi v_{\varepsilon}E\left[  \beta
_{0}^{^{\prime}}(u^{\varepsilon}_{t})r^{\varepsilon}_{t}\right]  $) it is easy
to see that equation (\ref{identity time derivative}) holds for $\varphi
(x)=\left(  1+\left\vert x\right\vert \right)  ^{-N}$. Moreover%

\begin{align*}
\int\varphi v_{\varepsilon}\Delta v_{\varepsilon}  &  =-\int\varphi\left\vert
\nabla v_{\varepsilon}\right\vert ^{2}-\int v_{\varepsilon}\nabla\varphi
\cdot\nabla v_{\varepsilon}\leq-\int\varphi\left\vert \nabla v_{\varepsilon
}\right\vert ^{2}+N\int\left\vert v_{\varepsilon}\right\vert \left\vert
\varphi\right\vert \left\vert \nabla v_{\varepsilon}\right\vert \\
&  \leq-\frac{1}{2}\int\varphi\left\vert \nabla v_{\varepsilon}\right\vert
^{2}+\frac{N^{2}}{2}\int\left\vert v_{\varepsilon}\right\vert ^{2}\left\vert
\varphi\right\vert
\end{align*}%
\[
\int\varphi v_{\varepsilon}b\cdot\nabla v_{\varepsilon}=\int\varphi
v_{\varepsilon}b_{1}\cdot\nabla v_{\varepsilon}+\int\varphi v_{\varepsilon
}b_{2}\cdot\nabla v_{\varepsilon}%
\]%
\[
-2\int\varphi v_{\varepsilon}b_{1}\cdot\nabla v_{\varepsilon}\leq2\left\Vert
b_{1}\left(  t\right)  \right\Vert _{L^{\infty}\left(  \mathbb{R}^{d}\right)
}\int\varphi\left\vert v_{\varepsilon}\right\vert \left\vert \nabla
v_{\varepsilon}\right\vert \leq\frac{1}{4}\int\varphi\left\vert \nabla
v_{\varepsilon}\right\vert ^{2}+C\left\Vert b_{1}\left(  t\right)  \right\Vert
_{L^{\infty}\left(  \mathbb{R}^{d}\right)  }^{2}\int\varphi\left\vert
v_{\varepsilon}\right\vert ^{2}dx
\]%
\begin{align*}
-2\int\varphi v_{\varepsilon}b_{2}\cdot\nabla v_{\varepsilon}  &
=-\int\varphi b_{2}\cdot\nabla v_{\varepsilon}^{2}=\int v_{\varepsilon}%
^{2}b_{2}\cdot\nabla\varphi+\int v_{\varepsilon}^{2}\varphi\operatorname{div}%
b_{2}\\
&  \leq\left\Vert \frac{b_{2}\left(  t\right)  }{1+\left\vert x\right\vert
}\right\Vert _{L^{\infty}\left(  \mathbb{R}^{d}\right)  }^{2}\int
v_{\varepsilon}^{2}\left(  1+\left\vert x\right\vert \right)  \left\vert
\nabla\varphi\left(  x\right)  \right\vert dx+\left\Vert \operatorname{div}%
b_{t}\right\Vert _{L^{\infty}\left(  \mathbb{R}^{d}\right)  }\int
v_{\varepsilon}^{2}\varphi dx\\
&  \leq\left(  \left\Vert \frac{b_{2}\left(  t\right)  }{1+\left\vert
x\right\vert }\right\Vert _{L^{\infty}\left(  \mathbb{R}^{d}\right)  }%
^{2}+\left\Vert \operatorname{div}b_{t}\right\Vert _{L^{\infty}\left(
\mathbb{R}^{d}\right)  }\right)  N\int v_{\varepsilon}^{2}\varphi dx.
\end{align*}
Summarizing,%
\[
\frac{d}{dt}\int\varphi\left\vert v_{\varepsilon}\right\vert ^{2}+\frac{1}%
{4}\int\varphi\left\vert \nabla v_{\varepsilon}\right\vert ^{2}\leq
C_{N}\alpha\left(  t\right)  \int\left\vert v_{\varepsilon}\right\vert
^{2}\varphi dx+C\left\Vert v\right\Vert _{L^{\infty}\left(  \left[
0,T\right]  \times\mathbb{R}^{d}\right)  }\int\varphi E\left[  \left|
\beta_{0}^{^{\prime}}(u^{\varepsilon}_{t})r^{\varepsilon}_{t}\right|  \right]
dx
\]
where%
\[
\alpha\left(  t\right)  :=\left\Vert b_{1}\left(  t\right)  \right\Vert
_{L^{\infty}\left(  \mathbb{R}^{d}\right)  }^{2}+\left\Vert \frac{b_{2}\left(
t\right)  }{1+\left\vert x\right\vert }\right\Vert _{L^{\infty}\left(
\mathbb{R}^{d}\right)  }^{2}+\left\Vert \operatorname{div}b_{t}\right\Vert
_{L^{\infty}\left(  \mathbb{R}^{d}\right)  }%
\]
is integrable. By Gronwall lemma and the result of the next step we deduce
\[
\lim_{\varepsilon\rightarrow0}\int\varphi\left(  x\right)  \left\vert
v_{\varepsilon}\left(  t,x\right)  \right\vert ^{2}dx=0
\]
for all $t\in\left[  0,T\right]  $, and thus $v=0$.

\textbf{Step 3}. It remains to show that
\[
\int^{T}_{0}\int(1+|x|)^{-N}E\left[  \left|  \beta^{^{\prime}}(u^{\varepsilon
}_{t})r^{\varepsilon}_{t}\right|  \right]  dxdt\rightarrow0
\]
First of all note that, given a convolution kernel $\rho_{\varepsilon}$, for
$\varepsilon$ sufficiently small it holds $\frac{1}{(1+|x|)^{N}}\ast
\rho_{\varepsilon}\leq2 \frac{1}{(1+|x|)^{N}}$. Therefore we can apply
proposition \ref{stima epsilon} and we obtain:
\begin{align}
\int^{T}_{0}\int_{|x|\geq R}(1+|x|)^{-N}|\beta^{^{\prime}}(u^{\varepsilon
})r_{\varepsilon}|dxdt\leq C\|\beta^{^{\prime}}(u)\|_{\infty}\|u\|_{\infty
}\int^{T}_{0}\int_{|x|\geq R}\frac{|Db_{t}|}{(1+|x|)^{N}}dxdt<\infty
\end{align}
Since
\[
\lim_{R\rightarrow+\infty} \int^{T}_{0}\int_{|x|\geq R-1}\frac{|Db|}%
{(1+|x|)^{N}}dxdt=0
\]
it is sufficient to prove that, for every $R>0$
\[
\lim_{\varepsilon\rightarrow0}\int^{T}_{0}\int_{B_{R}} E\left[  \left|
\beta^{^{\prime}}(u^{\varepsilon}_{t})r^{\varepsilon}_{t}\right|  \right]
dxdt= 0
\]
Thanks to the bound $\int^{T}_{0}\int_{B_{R}}\left|  \beta^{^{\prime}%
}(u^{\varepsilon}_{t})r^{\varepsilon}_{t}\right|  dxdt\leq C\|\beta^{^{\prime
}}(u)\|_{\infty}\|u\|_{\infty}\int^{T}_{0}\int_{B_{R+1}}|Db_{t}|dxdt$, which
holds a.s. for proposition \ref{stima epsilon}, is sufficient to prove that
for any sequence $\varepsilon_{n}$ exists a subsequence, still denoted by
$\varepsilon_{n}$ such that a.s. $\int^{T}_{0}\int_{B_{R}}\left|
\beta^{^{\prime}}(u^{\varepsilon}_{t})r^{\varepsilon}_{t}\right|
dxdt\rightarrow0$. This follows from the proof of the previous theorem, where
it was proved that, for every sequence $\varepsilon_{n}$ there exist a
subsequence still denoted by $\varepsilon_{n}$ such that $\beta^{^{\prime}%
}(u^{\varepsilon_{n}}_{t})r^{\varepsilon_{n}}_{t}$ a.s. converges to a measure
$\sigma$, with finite total variation on $[0,T]\times B_{R}$, and then it was
proved that a.s. $\sigma=0$.

\section{Remarks on a few variants\label{section variants}}

In a sentence, the core of the method is the commutator lemma (or
renormalizability of solutions) which requires classical assumptions on
$\left(  b,u\right)  $, plus a theorem of uniqueness of non-negative
$L^{\infty}$-solutions $v=E\left[  \beta\left(  u\right)  \right]  $ of the
parabolic equation (\ref{parabolic equation}). There are recent advanced
results on this parabolic equation under assumption on $b$ coherent with the
present framework, expecially \cite{Figalli} and \cite{LeBrisLions}. But they
do not fit precisely with our purposes for different reasons. For instance,
\cite{LeBrisLions} show that, due to the Laplacian, one can weaken the
assuptions on $b$, but the uniqueness is in the class of solutions $v$ with
the usual variational regularity (including $L^{2}\left(  0,T;W^{1,2}\left(
\mathbb{R}^{d}\right)  \right)  $). We do not know that $E\left[  \beta\left(
u\right)  \right]  $ has this regularity. The paper \cite{Figalli} deals with
only $L^{\infty}$-solutions $v$ of the parabolic equation
(\ref{parabolic equation}) and proves uniqueness, but again assuming
$\operatorname{div}b\in L^{\infty}$, the generalization of which is one of our
main purposes, otherwise the theory would be equal to the deterministic one.

This is the reason why we have given above a self-contained proof of
uniqueness for\ equation (\ref{parabolic equation}). There are other proofs,
under easier or different assumptions. We have given above, as a main theorem,
the one which allows us to deal with the $BV$ set-up and linear growth $b$. In
other directions of generality, or simplicity of proofs, we have the following
two results. We only sketch the proofs. The first theorem is a particular case
of theorem \ref{main theorem}, with $b=b_{1}$. We give a very simple proof by
semigroup theory, which could be generalized to other analytic semigroups in
$L^{1}\left(  \mathbb{R}^{d}\right)  $ different from the heat semigroup.

\begin{theorem}
If, in addition to hypothesis (\ref{assumption 1 on b}), we assume%
\[
Db\in L^{1}\left(  \left[  0,T\right]  \times\mathbb{R}^{d}\right)  ,\qquad
b\in L^{\infty}\left(  \left[  0,T\right]  \times\mathbb{R}^{d}\right)  .
\]
(which includes (\ref{weak differentiability assumption})), then there exists
a unique weak $L^{\infty}$-solution of equation (\ref{stoch transport eq}).
\end{theorem}

\begin{proof}
From the SPDE to the parabolic PDE (\ref{parabolic equation}) the proof is the
same as in theorem \ref{main theorem}. We have only to show uniqueness of the
solution $v$ to (\ref{parabolic equation}). Let $\theta_{\varepsilon}$ be the
mollifiers introduced above and let $v_{\varepsilon}\left(  t,\cdot\right)
=\theta_{\varepsilon}\ast v\left(  t,\cdot\right)  $ (convolution in space).
Take $\beta$ as in the proof of theorem \ref{main theorem}. We have%
\[
\frac{\partial v_{\varepsilon}}{\partial t}+b\cdot\nabla v_{\varepsilon}%
=\frac{1}{2}\Delta v_{\varepsilon}+r_{\varepsilon},\qquad v_{\varepsilon
}\left(  0\right)  =0
\]
where $r_{\varepsilon}$ is the usual commutator. Therefore%
\[
v_{\varepsilon}\left(  t\right)  =\int_{0}^{t}T_{t-s}\left(  r_{\varepsilon
}\left(  s\right)  -b\left(  s\right)  \cdot\nabla v_{\varepsilon}\left(
s\right)  \right)  ds
\]
where%
\[
\left(  T_{t}f\right)  \left(  x\right)  :=\int_{\mathbb{R}^{d}}f\left(
x+y\right)  \left(  2\pi t\right)  ^{-d/2}\exp\left(  -\frac{\left\vert
y\right\vert ^{2}}{2t}\right)  dy.
\]
Notice that $v_{\varepsilon}\in L^{\infty}\left(  \left[  0,T\right]
\times\mathbb{R}^{d}\right)  $, hence all the previous integrals are well defined.

We have
\[
\left\Vert r_{\varepsilon}\left(  t\right)  \right\Vert _{L^{1}\left(
\mathbb{R}^{d}\right)  }\leq C\left\Vert v\left(  t\right)  \right\Vert
_{L^{\infty}\left(  \mathbb{R}^{d}\right)  }\left\Vert b\left(  t\right)
\right\Vert _{W^{1,1}\left(  \mathbb{R}^{d}\right)  }%
\]
and $\left\Vert r_{\varepsilon}\left(  t\right)  \right\Vert _{L^{1}\left(
\mathbb{R}^{d}\right)  }\rightarrow0$ as $\varepsilon\rightarrow0$, see
\cite{PL Lions}, lemma 2.3. Moreover, the heat semigroup has the following
property:%
\[
\left\Vert \int_{0}^{t}T_{t-s}f\left(  s\right)  ds\right\Vert _{W^{1,1}%
\left(  \mathbb{R}^{d}\right)  }\leq\int_{0}^{t}\frac{C}{\left(  t-s\right)
^{1/2}}\left\Vert f\left(  s\right)  \right\Vert _{L^{1}\left(  \mathbb{R}%
^{d}\right)  }ds.
\]
This implies, with $f\left(  s\right)  =r_{\varepsilon}\left(  s\right)
-b\left(  s\right)  \cdot\nabla v_{\varepsilon}\left(  s\right)  $%
\begin{align*}
\int_{0}^{T}\left\Vert v_{\varepsilon}\left(  t\right)  \right\Vert
_{W^{1,1}\left(  \mathbb{R}^{d}\right)  }dt  & \leq\int_{0}^{T}\int_{0}%
^{t}\frac{C}{\left(  t-s\right)  ^{1/2}}\left\Vert f\left(  s\right)
\right\Vert _{L^{1}\left(  \mathbb{R}^{d}\right)  }dsdt\\
& \leq C\sqrt{T}\int_{0}^{T}\left\Vert f\left(  s\right)  \right\Vert
_{L^{1}\left(  \mathbb{R}^{d}\right)  }ds
\end{align*}
and thus%
\[
\int_{0}^{T}\left\Vert v_{\varepsilon}\left(  t\right)  \right\Vert
_{W^{1,1}\left(  \mathbb{R}^{d}\right)  }dt\leq C\sqrt{T}\int_{0}%
^{T}\left\Vert r_{\varepsilon}\left(  t\right)  \right\Vert _{L^{1}\left(
\mathbb{R}^{d}\right)  }dt+C\sqrt{T}\int_{0}^{T}\left\Vert b\right\Vert
_{L^{\infty}}\left\Vert v_{\varepsilon}\left(  t\right)  \right\Vert
_{W^{1,1}\left(  \mathbb{R}^{d}\right)  }dt.
\]
For small $T>0$ this gives us%
\[
\int_{0}^{T}\left\Vert v_{\varepsilon}\left(  t\right)  \right\Vert
_{W^{1,1}\left(  \mathbb{R}^{d}\right)  }dt\leq C_{T}\int_{0}^{T}\left\Vert
r_{\varepsilon}\left(  t\right)  \right\Vert _{L^{1}\left(  \mathbb{R}%
^{d}\right)  }dt
\]
and thus (by the properties of $r_{\varepsilon}$ recalled above and Lebesgue
theorem) $\lim_{\varepsilon\rightarrow0}\int_{0}^{T}\left\Vert v_{\varepsilon
}\left(  t\right)  \right\Vert _{W^{1,1}\left(  \mathbb{R}^{d}\right)  }dt=0$.
This implies $v=0$. The proof is complete.
\end{proof}

Next theorem is a little generalization of theorem \ref{main theorem} in the
direction of the so called Prodi-Serrin condition of fluid dynamics. The
stronger condition $\frac{2}{q}+\frac{d}{p}<1$ is the basic one in the work
\cite{Kry-Ro}. Under the same assumption it has been proved in \cite{FedFla}
that the stochastic characteristics generate a flow of H\"{o}lder
homeomorphisms, so probably this assumption could be related to a future
generalization of the approach of \cite{FGP}.

\begin{theorem}
Theorem \ref{main theorem} remains true if we replace assumption $b_{1}\in
L^{2}\left(  0,T;L^{\infty}\left(  \mathbb{R}^{d}\right)  \right)  $ with%
\[
b\in L^{q}\left(  0,T;L^{p}\left(  \mathbb{R}^{d}\right)  \right)
,\qquad\frac{2}{q}+\frac{d}{p}\leq1,\qquad p,q\in(1,\infty).
\]

\end{theorem}

\begin{proof}
We have only to modify the estimate for $\int\varphi v_{\varepsilon}b_{1}%
\cdot\nabla v_{\varepsilon}$. Here we use the following bounds:
\[
\left\vert \int\varphi v_{\varepsilon}b_{1}\cdot\nabla v_{\varepsilon
}dx\right\vert \leq\frac{1}{4}\int\varphi\left\vert \nabla v_{\varepsilon
}\right\vert ^{2}dx+C^{\ast}\int\varphi\left\vert b_{1}\right\vert
^{2}v_{\varepsilon}^{2}dx
\]
\[
\int\varphi\left\vert b_{1}\right\vert ^{2}v_{\varepsilon}^{2}dx\leq\left(
\int\left\vert b_{1}\right\vert ^{p}dx\right)  ^{2/p}\left(  \int\left\vert
\sqrt{\varphi}v_{\varepsilon}\right\vert ^{2p/\left(  p-2\right)  }dx\right)
^{\frac{p-2}{p}}%
\]
and by an interpolation inequality%
\begin{align*}
& \leq C\left(  \int\left\vert b_{1}\right\vert ^{p}dx\right)  ^{2/p}%
\left\Vert \sqrt{\varphi}v_{\varepsilon}\right\Vert _{W^{\gamma,2}}^{2}\\
& \leq C\left(  \int\left\vert b_{1}\right\vert ^{p}dx\right)  ^{2/p}%
\left\Vert \sqrt{\varphi}v_{\varepsilon}\right\Vert _{L^{2}}^{2-2\gamma
}\left\Vert \sqrt{\varphi}v_{\varepsilon}\right\Vert _{W^{1,2}}^{2\gamma}\\
& \leq\frac{1}{8C^{\ast}}\left\Vert \sqrt{\varphi}v_{\varepsilon}\right\Vert
_{W^{1,2}}^{2}+C\left(  \int\left\vert b_{1}\right\vert ^{p}dx\right)
^{\frac{2}{p}\cdot\frac{1}{1-\gamma}}\left\Vert \sqrt{\varphi}v_{\varepsilon
}\right\Vert _{L^{2}}^{2}%
\end{align*}
where $\gamma<1$, $\frac{p-2}{2p}=\frac{1}{2}-\frac{\gamma}{d}$, namely
$\gamma=\frac{d}{p}$. We have%
\[
\left\Vert \sqrt{\varphi}v_{\varepsilon}\right\Vert _{W^{1,2}}^{2}\leq
\int\varphi v_{\varepsilon}^{2}dx+2\int\left\vert v_{\varepsilon}\nabla
\sqrt{\varphi}\right\vert ^{2}dx+2\int\varphi\left\vert \nabla v_{\varepsilon
}\right\vert ^{2}dx
\]
and, recalling that $\left\vert \nabla\varphi\right\vert \leq N\left\vert
\varphi\right\vert $ and $\left\vert \varphi\right\vert \leq1$,
\[
\int\left\vert v_{\varepsilon}\nabla\sqrt{\varphi}\right\vert ^{2}dx\leq
\frac{1}{2}\int v_{\varepsilon}^{2}\frac{\left\vert \nabla\varphi\right\vert
^{2}}{\sqrt{\varphi}}dx\leq\frac{1}{2}\int v_{\varepsilon}^{2}\frac
{N^{2}\left\vert \varphi\right\vert ^{2}}{\sqrt{\varphi}}dx\leq\frac{N^{2}}%
{2}\int v_{\varepsilon}^{2}\varphi dx.
\]
Therefore%
\[
\left\Vert \sqrt{\varphi}v_{\varepsilon}\right\Vert _{W^{1,2}}^{2}\leq
C_{N}\int v_{\varepsilon}^{2}\varphi dx+2\int\varphi\left\vert \nabla
v_{\varepsilon}\right\vert ^{2}dx.
\]
Summarizing, we have proved that%
\[
\left\vert \int\varphi v_{\varepsilon}b_{1}\cdot\nabla v_{\varepsilon
}dx\right\vert \leq\frac{1}{4}\int\varphi\left\vert \nabla v_{\varepsilon
}\right\vert ^{2}dx+C_{N}\left(  1+\left\Vert b_{1}\right\Vert _{L^{p}\left(
\mathbb{R}^{d}\right)  }^{\frac{2}{1-\gamma}}\right)  \int v_{\varepsilon}%
^{2}\varphi dx
\]
for a suitable constant $C_{N}$. It is now easy to complete the proof of
theorem \ref{main theorem} by Gronwall lemma, if we check that%
\[
\frac{2}{1-\gamma}\leq q.
\]
Since $\gamma=\frac{d}{p}$, the inequality is $\frac{2}{q}\leq1-\gamma
=1-\frac{d}{p}$, which is preisely our assumption. The proof is complete.
\end{proof}

\section{Example\label{section example}}

We give a simple example, with the flavour of shear flows, which is covered by
theorem \ref{main theorem} but apparently not by the previous deterministic or
stochastic works.

Consider $d=2$ and
\[
b\left(  x,y\right)  =sign\left(  y\right)  \left(
\begin{array}
[c]{c}%
1\\
2\sqrt{\left\vert y\right\vert }%
\end{array}
\right)  .
\]
This is not covered by \cite{DiPernaLions} or \cite{Ambrosio} because
\[
\operatorname{div}b=\frac{1}{\sqrt{\left\vert y\right\vert }}%
\]
is not bounded, and not by \cite{FGP} because $b$ is discontinuous. Without
noise, the equations of characteristics%
\begin{align*}
X^{\prime}  & =sign\left(  Y\right)  \\
Y^{\prime}  & =2sign(Y)\sqrt{\left\vert Y\right\vert }%
\end{align*}
have multiple solutions from every initial condition on the line $\left(
x_{0},0\right)  $, $x_{0}\in\mathbb{R}$, the ideal surface of separation of
this \textquotedblleft compressible shear flow\textquotedblright, which move
away fron the surface as in a sort of instability process. Using these
multiple solutions one can write down multiple solutions of the deterministic
trasport equation, with any initial condition $u_{0}$. On the contrary, the
stochastic equation (\ref{stoch transport eq}) is well posed, since we may
apply theorem \ref{main theorem} with%
\[
b_{1}\left(  x,y\right)  =sign\left(  y\right)  \left(
\begin{array}
[c]{c}%
1\\
2\sqrt{\left\vert y\right\vert \wedge1}%
\end{array}
\right)
\]
and $b_{2}=b-b_{1}$. We have $\partial_{y}b^{\left(  1\right)  }\left(
x,y\right)  =2\delta\left(  y\right)  $, integrable with $N=2$ in the sense of
the assumptions of theorem \ref{main theorem};\ and $Db_{2}$ bounded, so it is
easy to see that the assumptuions of the main theorem are fulfilled and we
have uniqueness.

\end{document}